\documentclass[10pt]{amsart}
\usepackage{amsmath, latexsym, amsthm, amsfonts,bm,amssymb} 
\usepackage{natbib} 
\usepackage{url} 
\usepackage{enumerate}

\bibpunct{(}{)}{;}{a}{}{,} 

\theoremstyle{plain}
\newtheorem{thm}{Theorem}

\newtheorem{lemma}{Lemma}

\newtheorem{defin}{Definition}
\newtheorem{assump}{Assumption}

\theoremstyle{remark}
\newtheorem{rem}{Remark}

\def\convp{\xrightarrow{\mathbb{P}_{s_0}}}

\def\ex{{\rm {\mathbb{E}_{s_0}\,}}}

\def\pp{{\rm {\mathbb{P}_{s_0}\,}}}

\allowdisplaybreaks

\begin{document}

\title[Bayesian dispersion coefficient estimation]{Posterior contraction rate for non-parametric Bayesian estimation of the dispersion coefficient of a stochastic differential equation}

\author{Shota Gugushvili}
\address{Mathematical Institute\\
Leiden University\\
P.O. Box 9512\\
2300 RA Leiden\\
The Netherlands}
\email{shota.gugushvili@math.leidenuniv.nl}

\author{Peter Spreij}
\address{Korteweg-de Vries Institute for Mathematics\\
University of Amsterdam\\
PO Box 94248\\
1090 GE Amsterdam\\
The Netherlands}
\email{spreij@uva.nl}



\begin{abstract}

We derive the posteror contraction rate for non-parametric Bayesian estimation of a deterministic dispersion coefficient of a linear stochastic differential equation.

\end{abstract}

\date{\today}

\maketitle

\section{Introduction}
\label{intro}

Suppose a simple linear stochastic differential equation
\begin{equation}
\label{sde}
dX_t=s(t)dW_t, \quad X_0=x, \quad t\in[0,1],
\end{equation}
with a deterministic dispersion coefficient $s$ and a deterministic initial condition $X_0=x$ is given. Here $W$ is a Brownian motion. Without loss of generality we take $x=0.$ The process $X$ is Gaussian with mean zero and covariance
$
\rho(u,v)=\int_0^{u \wedge v} (s(t))^2dt.
$
By $\mathbb{P}_{s}$ we will denote the law of the process $X$ corresponding to the dispersion coefficient $s$ in \eqref{sde}. The dispersion coefficient $s$ in \eqref{sde} can be interpreted as a signal passing through a noisy channel, where the noise is multiplicative and is modelled by the Brownian motion.

Suppose that corresponding to the true dispersion coefficient $s=s_0$ in \eqref{sde}, a sample $X_{t_{i,n}},i=1,\ldots,n,$ from the process $X$ is at our disposal, where $t_{i,n}=i/n,i=0,\ldots,n.$ Our goal is non-parametric Bayesian estimation of $s_0.$ Related references employing the frequentist approach for a similar model are \cite{genon92}, \cite{hoffmann97} and \cite{soulier98}. For a Bayesian approach see \cite{gugu}. Note that our model shows obvious similarities to a standard non-parametric regression model, or to the white noise model (see e.g.\ \cite{rasmussen06} or \cite{vaart08} for these models in the non-parametric Bayesian context), but also possesses distinctive features of its own.

Let $\mathcal{X}$ denote some non-parametric class of dispersion coefficients $s.$ The likelihood corresponding to the observations $X_{t_{i,n}}$ is given by
\begin{equation}
\label{likelih}
L_n(s)=\prod_{i=1}^{n} \left\{ \frac{1}{\sqrt{2\pi\int_{t_{i-1,n}}^{t_{i,n}}s^2(u)du}}\psi\left( \frac{X_{t_{i,n}}-X_{t_{i-1,n}}}{\sqrt{\int_{t_{i-1,n}}^{t_{i,n}}s^2(u)du}} \right) \right\},
\end{equation}
where $\psi(u)=\exp(-u^2/2).$ For a prior $\Pi$ on $\mathcal{X},$ the posterior measure of any measurable set $\mathcal{S}\subset\mathcal{X}$ can be obtained through Bayes' formula,
\begin{equation*}
\Pi(\mathcal{S}|X_{t_{0,n}}\ldots,X_{n,n})=\frac{\int_{\mathcal{S}} L_n(s) \Pi(d s)}{ \int_{\mathcal{X}} L_n(s) \Pi(ds) }.
\end{equation*}
One can then proceed with the computation of other quantities of interest in the Bayesian paradigm, for instance point estimates of $s_0,$ credible sets and so on.

A desirable property of a Bayes procedure is posterior consistency. In our context posterior consistency  means that for every neighbourhood $U_{s_0}$ of $s_0$ (in a suitable topology)
\begin{equation*}
\Pi(U_{s_0}^c|X_{t_{0,n}},\ldots,X_{t_{n,n}})\convp 0
\end{equation*}
as $n\rightarrow\infty.$ In other words, when viewed under the true law $\pp,$ a consistent Bayesian procedure asymptotically puts posterior mass equal to one on every fixed neighbourhood of the true parameter $s_0.$ Study of posterior consistency is similar to study of consistency of frequentist estimators, and in fact, if posterior consistency holds, the center of the posterior distribution (in an appropriate sense) will provide a consistent (in the frequentist sense) estimator of the parameter of interest.  For an introduction to consistency issues in Bayesian non-parametric statistics, see e.g.\ \cite{ghosal99} and \cite{wasserman98}. Posterior consistency for the model \eqref{sde} was shown in \cite{gugu}.

More generally, instead of a fixed neighbourhood $U_{s_0}$ of the true parameter $s_0,$ one can also take a sequence of neighbourhoods $U_{s_0,\varepsilon_n}$ shrinking to $s_0$ at a rate $\varepsilon_n\rightarrow 0$ (the sequence $\varepsilon_n$ determines the size of the neighbourhood) and ask at what rate is $\varepsilon_n$ allowed to decay, so that the neighbourhoods $U_{s_0,\varepsilon_n}$ still manage to capture most of the posterior mass. A formal way to state this is
\begin{equation}
\label{convergence}
\Pi(U_{s_0,\varepsilon_n}^c|X_{t_{0,n}},\ldots,X_{t_{n,n}})\convp 0
\end{equation}
as $n\rightarrow\infty.$ The rate $\varepsilon_n$ is called the posterior contraction rate, or the posterior convergence rate. Note that $\varepsilon_n$ is not uniquely defined: if $\varepsilon_n$ is a posterior contraction rate, then so is e.g.\ $2\varepsilon_n,$ because $ U_{s_0,2\varepsilon_n}^c \subset U_{s_0,\varepsilon_n}^c.$ This, however, is true also for the convergence rate of frequentist estimators, cf.\ a discussion on p.\ 79 in \cite{tsyb}. In general we are interested in determination of the `fastest' rate of decay of $\varepsilon_n,$ so that \eqref{convergence} still holds. Some general references on derivation of posterior convergence rates under various statistical setups are  \cite{ghosal00}, \cite{ghosal07} and \cite{shen01}. Study of this question parallels the analysis of convergence rates of various estimators in the frequentist literature. In fact, a property like \eqref{convergence} also implies that Bayes point estimates have the convergence rate $\varepsilon_n$ (in the frequentist sense), cf.\ pp.\ 506--507 in \cite{ghosal00}.  It is well-known that in finite-dimensional statistical problems under suitable regularity assumptions Bayes procedures yield optimal (in the frequentist sense) estimators. The situation is much more subtle in the infinite-dimensional setting: a careless choice of the prior might violate posterior consistentsy, or the posterior might concentrate around the true parameter value at a suboptimal rate (here by `suboptimal' we mean the rate slower than the minimax rate for estimation of $s_0$). Hence the importance of derivation of the posterior contraction rate. 

The rest of the paper is organised as follows: in Section \ref{result} we formulate a theorem establishing \eqref{convergence} under suitable conditions. Section \ref{discussion} contains a brief discussion on the obtained result. The proof of the theorem is given in Section \ref{proof}, while the Appendix contains a number of technical lemmas used in the proof of the theorem.

Throughout the paper we will use the following notation to compare two sequences $a_n$ and $b_n$ of real numbers: $a_n\lesssim b_n$ will mean that there exists a constant $B>0$ that is independent of $n$ and is such that $a_n\leq B b_n;$ $a_n\gtrsim b_n$ will mean that there exists a constant $A>0$ that is independent of $n$ and is such that $A a_n\geq b_n;$ $a_n\asymp b_n$ will mean that $a_n$ and $b_n$ are asymptotically of the same order, i.e.\ $-\infty<\liminf_{n\rightarrow\infty} a_n/b_n\leq \limsup_{n\rightarrow\infty} a_n/b_n<\infty.$

\section{Main theorem}
\label{result}

We first specify the non-parametric class $\mathcal{X}$ of dispersion coefficients $s.$

\begin{defin}
\label{classX} Let $\mathcal{X}$ be the collection of dispersion coefficients $s:[0,1]\rightarrow[\kappa,\mathcal{K}],$ such that $s\in\mathcal{X}$ is differentiable and $\|s^{\prime}\|_{\infty}\leq M.$ Here $0<\kappa<\mathcal{K}<\infty$ and $0<M<\infty$ are three constants independent of a particular $s\in\mathcal{X},$ while $\|\cdot\|_{\infty}$ denotes the $L_{\infty}$-norm.
\end{defin}

\begin{rem}
\label{constrem}
Since $\mathbb{P}_{s}=\mathbb{P}_{-s},$ a positivity assumption on $s\in\mathcal{X}$ in Definition \ref{classX} is a natural identifiability requirement. Furthermore, strict positivity of $s$ allows one to avoid complications when manipulating the likelihood \eqref{likelih}. Boundedness and differentiability of $s$ also come in handy in the proof of Theorem \ref{mainthm} below. \qed
\end{rem}

We summarise the assumptions on our statistical model.

\begin{assump}
\label{standing}
Assume that
\begin{enumerate}[(a)]
\item the model \eqref{sde} is given with $x=0$ and $s\in\mathcal{X},$ where $\mathcal{X}$ is defined in Definition \ref{classX},
\item $s_0\in\mathcal{X}$ denotes the true dispersion coefficient,
\item a discrete-time sample $\{X_{t_{i,n}}\}$ from the solution $X$ to \eqref{sde} corresponding to $s_0$ is available, where $t_{i,n}=i/n,i=0,\ldots,n.$
\end{enumerate}
\end{assump}

For $\varepsilon>0$ introduce the notation
\begin{equation*}
U_{s_0,\varepsilon}=\left\{ s\in\mathcal{X}:\|s-s_0\|_2<\varepsilon \right\}, \quad V_{s_0,\varepsilon}=\left\{ {s}\in\mathcal{X}: {\| {s}-s_0\|_{\infty}} <\varepsilon \right\}.
\end{equation*}
Here $\|\cdot\|_2$ denotes the $L_2$-norm. We will establish \eqref{convergence} for the complements of the neighbourhoods $U_{s_0,\varepsilon_n}$ of the true parameter $s_0$ and determine the corresponding posterior contraction rate $\varepsilon_n.$

\begin{thm}
\label{mainthm}
Suppose that Assumption \ref{standing} holds. Let the sequence $\widetilde{\varepsilon}_n$ of positive numbers be such that $\widetilde{\varepsilon}_n\asymp n^{-1/3}\log n,$ and let the prior $\Pi$ on $\mathcal{X}$ be such that
\begin{equation}
\label{priorC}
\Pi( V_{s_0,\widetilde{\varepsilon}_n} )  \gtrsim e^{- \overline{C} n \widetilde{\varepsilon}_n^2}
\end{equation}
for some constant $\overline{C}>0$ that is independent of $n.$ Then for a large enough constant $\widetilde{M}$ and a sequence $\varepsilon_n = \widetilde{M} \widetilde{\varepsilon}_n,$
\begin{equation*}
\Pi(U_{s_0,\varepsilon_n}^c|X_{t_{0,n}},\ldots,X_{t_{n,n}})\convp 0
\end{equation*}
holds.
\end{thm}

\begin{rem}
\label{rem1}
An essential condition in Theorem \ref{mainthm} is \eqref{priorC}. A prior $\Pi$ satisfying condition \eqref{priorC} can be constructed, for instance, through a construction similar to the one given in Section 3 of \cite{ghosal00}, that is based on finite approximating sets (this type of prior was introduced in \cite{ghosal97}). \qed
\end{rem}

\begin{rem}
\label{rem2}
Theorem \ref{mainthm} can be generalised to the case where the members of the class $\mathcal{X}$ of dispersion coefficients are $\beta\geq 1$ times differentiable with derivatives satisfying suitable boundedness assumptions. The convergence rate that can be obtained in that case is (up to a logarithmic factor) $n^{-\beta/(2\beta+1)}.$  \qed
\end{rem}

\section{Discussion}
\label{discussion}

Theorem \ref{mainthm} states that under the differentiability assumption on the members $s$ of the class $\mathcal{X}$ of dispersion coefficients, the posterior contracts around the true dispersion coefficient $s_0$ at the rate $n^{-1/3}\log n.$ This implies existence of Bayes estimates that converge (in the frequentist sense) to $s_0$ at the same rate. By Proposition 1 from \cite{hoffmann97}, the rate $n^{-1/3}$ is the minimax convergence rate for estimation of the diffusion coefficient $s_0^2$ with $L_2$-loss function in essentially the same model as ours. In this sense the rate derived in Theorem \ref{mainthm} can be thought of as essentially (up to a logarithmic factor) optimal posterior contraction rate. Whether the logarithmic factor is essential, or is just an artifact of our proof, is not entirely clear.

We would also like to make a brief comment on the proof of Theorem \ref{mainthm}: in principle, it is conceivable that its statement could be derived from some general result on the posterior contraction rate, see e.g.\ Sections 2 and 3 in \cite{ghosal07}. However, we take an alternative approach, that is similar in some respects to the one in \cite{shen01} and that relies on results from empirical process theory (see e.g.\ \cite{geer00}). This alternative approach is not necessarily the shortest or simplest, and the choice of a specific path to the derivation of a posterior convergence rate is perhaps a matter of taste.

\section{Proof of Theorem \ref{mainthm}}
\label{proof}
Throughout this section and the Appendix, $R_n(s)=L_n(s)/L_n(s_0)$ will denote the likelihood ratio corresponding to the observations $X_{t_i,n}.$  We will use the notation $P_{i,n,s}$ to denote the law of $Y_{i,n}=X_{t_{i,n}}-X_{t_{i-1,n}}$ corresponding to the parameter value $s$ in \eqref{sde} and $P_{i,n,0}$ to denote the law of $Y_{i,n}$ corresponding to the true parameter value $s_0$ in \eqref{sde}. The corresponding densities will be denoted by $p_{i,n,s}$ and $p_{i,n,0}.$ We also set
\begin{equation*}
z_i=t_{i-1,n},\quad\mathcal{W}_i=1-\frac{Y_{i,n}^2}{\int_{t_{i-1,n}}^{t_{i,n}}s_0^2(u)du}, \quad f_s(z)=\frac{\int_{z}^{z+1/n}[s_0^2(u)-s^2(u)]du}{\int_{z}^{z+1/n}s^2(u)du}.
\end{equation*}
The latter notation is reminiscent of the one used in \cite{geer00}. Note that the $\mathcal{W}_i$'s are i.i.d.\ with zero mean and variance equal to two.
\begin{proof}[Proof of Theorem \ref{mainthm}]
We have
\begin{equation*}
\Pi({U}_{s_0,\varepsilon_n}^c|X_{t_{0,n}}\ldots,X_{t_{n,n}})=\frac{\int_{{U}_{s_0,\varepsilon_n}^c} L_n(s) \Pi(ds)}{ \int_{\mathcal{X}} L_n(s) \Pi(ds) }
=\frac{\int_{{U}_{s_0,\varepsilon_n}^c} R_n(s) \Pi(ds)}{ \int_{\mathcal{X}} R_n(s) \Pi(ds) }=\frac{N_n}{D_n}.
\end{equation*}
We will establish the theorem by separately bounding $D_n$ and $N_n$ and then combining the bounds.

Let $S_n(s)=n^{-1}\log R_n(s).$ Then
$
D_n=\int_{\mathcal{X}} \exp( n S_n(s)) \Pi(ds).
$
We have
\begin{equation*}
S_n(s)=\frac{1}{2}\frac{1}{n}\sum_{i=1}^{n} \mathcal{W}_i f_s(z_{i})+\frac{1}{2}\frac{1}{n}\sum_{i=1}^{n}\left[  \log\left( 1+f_s(z_{i}) \right) - f_s(z_{i}) \right].
\end{equation*}
Let $n$ be large enough and assume that $s\in V_{s_0,\widetilde{\varepsilon}_n}.$ As a consequence of Lemmas \ref{lemma0} and \ref{lemma1} from the Appendix and by condition \eqref{priorC} on the prior, we get that with probability tending to one as $n\rightarrow\infty,$
\begin{equation}
\label{Dn}
\frac{1}{D_n}\leq \left( \int_{V_{s_0,\widetilde{\varepsilon}_n}} R_n(s)\Pi(ds) \right)^{-1} \lesssim \exp\left(\left( \frac{8\mathcal{K}^2}{\kappa^4} + \overline{C} \right)n\widetilde{\varepsilon}_n^2 \right).
\end{equation}
This finishes derivation of a bound for $D_n.$ We now turn to $N_n.$ In Lemma \ref{lemma2} from the Appendix we show that with probability tending to one as $n\rightarrow\infty,$ for some constant $c_1>0$ we have $N_n\leq\exp(-c_1n\varepsilon_n^2).$ Combination of this bound with \eqref{Dn} gives that with probability tending to one as $n\rightarrow\infty,$ the inequality
\begin{equation*}
\Pi({U}_{s_0,\varepsilon_n}^c|X_{t_{0,n}}\ldots,X_{t_{n,n}})\lesssim \exp\left( - c_1n\varepsilon_n^2 + \left( \frac{8\mathcal{K}^2}{\kappa^4} + \overline{C} \right)n\widetilde{\varepsilon}_n^2 \right)
\end{equation*}
is valid. From this it immediately follows that for $\varepsilon_n=\widetilde{M}\widetilde{\varepsilon}_n$ with a large enough constant $\widetilde{M},$ the left-hand side of the above display converges to zero in probability. This completes the proof of the theorem. 
\end{proof}

\appendix
\section*{Appendix}

Throughout the Appendix we will use the following notation: for any $\varepsilon>0,$ $M_{\varepsilon}$ will denote the smallest positive integer, such that $2^{M_{\varepsilon}}\varepsilon^2\geq 4 \mathcal{K}^2.$ Note that by definition $2^{M_{\varepsilon}}\varepsilon^2\leq 8\mathcal{K}^2,$ and that for $\varepsilon\rightarrow 0$ we have $M_{\varepsilon}\asymp \log_2(1/{\varepsilon}).$ We set $A_{j,\varepsilon}=\{ s\in\mathcal{X}: 2^j\varepsilon^2\leq \|s-s_0\|_2^2 < 2^{j+1}\varepsilon^2 \}$ and $B_{j,\varepsilon}=\{ s\in\mathcal{X}: \|s-s_0\|_2^2 < 2^{j+1}\varepsilon^2 \}$ for $j=0,1,\ldots,M_{\varepsilon}.$ We will also let $Z_{i,n,s}(Y_{i,n})=\log (p_{i,n,s}(Y_{i,n})/p_{i,n,0}(Y_{i,n}))$ denote the log-likelihood corresponding to one `observation' $Y_{i,n}.$ 

\begin{lemma}
\label{lemma0}
Let the conditions of Theorem \ref{mainthm} hold. Then
\begin{equation*}
\sup_{f_s\in\mathcal{F}_{s_0,\widetilde{\varepsilon}_n}}\left| \frac{1}{n} \sum_{i=1}^n \mathcal{W}_i f_s(z_{i}) \right| = O_{\mathbb{P}_{s_0}} \left( \delta_n \right),
\end{equation*}
where $\mathcal{F}_{s_0,\widetilde{\varepsilon}_n}=\{   f_s:\|s-s_0\|_{\infty}<\widetilde{\varepsilon}_n   \}$ and $\delta_n$ is an arbitrary sequence of positive numbers, such that $\delta_n \asymp {\widetilde{\varepsilon}_n}^2.$
\end{lemma}

\begin{proof}
We will establish the lemma using empirical process theory. In particular, we will employ Corollary 8.8 from \cite{geer00}. In light of the fact that $\widetilde{\varepsilon}_n\asymp n^{-1/3}\log n,$ in order to prove the lemma it suffices to show that
\begin{equation*}
\sup_{g_s\in\mathcal{G}_{s_0,\widetilde{\varepsilon}_n}}\left| \frac{1}{n} \sum_{i=1}^n \mathcal{W}_i g_s(z_i) \right|=O_{\mathbb{P}_{s_0}} \left( \delta_n \right),
\end{equation*}
where
\begin{equation*}
g_s(z)=\frac{s_0^2(z)-s^2(z)}{s^2(z)}, \quad \mathcal{G}_{s_0,\widetilde{\varepsilon}_n}=\{g_s: \|s-s_0\|_{\infty}<\widetilde{\varepsilon}_n \},
\end{equation*}
and the notation resembles the one in \cite{geer00}, so that the arguments become more transparent. Indeed, it suffices to note that by Assumption \ref{standing} (a) we have
$f_s(z_i)=g_s(z_i)+O(n^{-1}),$ whence
\begin{equation*}
\sup_{f_s\in\mathcal{F}_{s_0,\widetilde{\varepsilon}_n}}\left| \frac{1}{n} \sum_{i=1}^n \mathcal{W}_i f_s(z_i) \right| \leq \sup_{g_s\in\mathcal{G}_{s_0,\widetilde{\varepsilon}_n}}\left| \frac{1}{n} \sum_{i=1}^n \mathcal{W}_i g_s(z_i) \right| + O_{\mathbb{P}_{s_0}}\left(\frac{1}{n}\right).
\end{equation*}
In order to apply Corollary 8.8 from \cite{geer00}, we need to verify its conditions, and in particular we need to check formulae (8.23)--(8.29) there. This involves somewhat lengthy computations. Firstly, we need to find a constant $R_n,$ such that $\sup_{g_s\in\mathcal{G}_{s_0,\widetilde{\varepsilon}_n}}\|g_s\|_{Q_n}^2\leq R_n^2.$ Here $Q_n=n^{-1}\sum_{i=1}^n \delta_{z_i}$ is the empirical measure associated with the points $z_i$ and $\|g_s\|_{Q_n}^2=n^{-1}\sum_{i=1}^n g_s^2(z_i).$ Now,
$\|g_s\|_{Q_n}^2 \leq {4 \mathcal{K}^2}\widetilde{\varepsilon}_n^2/{\kappa^4}$ for $g_s\in \mathcal{G}_{s_0,\widetilde{\varepsilon}_n},$
and thus it suffices to take $R_n=2\mathcal{K}\widetilde{\varepsilon}_n/\kappa^2.$ Next, set $K_1=3.$ Using the rough bound $|e^x-1-x|\leq x^2 e^{|x|},$ we get that
\begin{equation*}
2K_1^2\ex\left[ e^{|\mathcal{W}_i|/K_1} - 1 - \frac{|\mathcal{W}_i|}{K_1} \right] \leq 2 \ex\left[ \mathcal{W}_i^2 e^{|\mathcal{W}_i|/3} \right]<\infty.
\end{equation*}
Let $\sigma_0^2=2 \ex\left[ \mathcal{W}_i^2 e^{|\mathcal{W}_i|/3} \right].$ With these $K_1$ and $\sigma_0,$ (8.23) in \cite{geer00} will be satisfied. Next we need to find a constant $K_2,$ such that the inequality $\sup_{g_s\in\mathcal{G}_{s_0,\widetilde{\varepsilon}_n}} \|g_s\|_{\infty}\leq K_2$ holds. One can take $K_2=2\mathcal{K}\widetilde{\varepsilon}_n/\kappa^2,$ and this verifies (8.24) in \cite{geer00}. We take $C_1=3,$ set $K=4K_1 K_2,$ and note that for all $n$ large enough,
$\delta_n \leq C_1 2 R_n^2 \sigma_0^2 / K$ and $ \delta_n \leq 8 \sqrt{2}R_n\sigma_0$ holds, because $\widetilde{\varepsilon}_n\rightarrow 0.$ This choice of $C_1$ and $K$ thus yields (8.25)--(8.27) in \cite{geer00}. Next let $C_0=2C,$ where $C$ is a universal constant as in Corollary 8.8 in \cite{geer00}. This choice of $C_0$ yields (8.29) in \cite{geer00}. It remains to check (8.28) in \cite{geer00}, i.e.\
\begin{equation}
\label{8.28}
\sqrt{n}\delta_n \geq C_0 \left( \int_{0}^{\sqrt{2}R_n\sigma_0} H_B^{1/2}\left( \frac{u}{\sqrt{2}\sigma_0},\mathcal{G}_{s_0,\widetilde{\varepsilon}_n},Q_n\right) du \vee \sqrt{2} R_n \sigma_0 \right),
\end{equation}
where $H_B\left( \delta,\mathcal{G}_{s_0,\widetilde{\varepsilon}_n},Q_n\right)$ is the $\delta$-entropy with bracketing of $\mathcal{G}_{s_0,\widetilde{\varepsilon}_n}$ for the $L_2(Q_n)$-metric (see Definition 2.2 in \cite{geer00}), and $a \vee b$ denotes the maximum of two numbers $a$ and $b.$ By Lemma 2.1 in \cite{geer00}, $H_B\left( \delta,\mathcal{G}_{s_0,\widetilde{\varepsilon}_n},Q_n\right) \leq H_{\infty}(\delta/2,\mathcal{G}_{s_0,\widetilde{\varepsilon}_n}),$ where $H_{\infty}(\delta,\mathcal{G}_{s_0,\widetilde{\varepsilon}_n})$ is the $\delta$-entropy of $\mathcal{G}_{s_0,\widetilde{\varepsilon}_n}$ for the supremum norm (see Definition 2.3 in \cite{geer00}). Lemma 3.9 in \cite{geer00} implies that for all $n$ large enough there exists a constant $A_1>0,$ such that $H_{\infty}(\delta,\mathcal{G}_{s_0,\widetilde{\varepsilon}_n})\leq A_1 \delta^{-1}$ for all $\delta>0$ (the fact that the matrix $\Sigma_{Q_n}$ from the statement of that lemma is non-singular can be shown by a minor variation of an argument from the proof of Lemma 1.4 in \cite{tsyb}). Hence
\begin{multline*}
\int_{0}^{\sqrt{2}R_n\sigma_0} H_B^{1/2}\left( \frac{u}{\sqrt{2}\sigma_0},\mathcal{G}_{s_0,\widetilde{\varepsilon}_n},Q_n\right) du\\ \leq \sqrt{A_1} \int_{0}^{\sqrt{2}R_n\sigma_0} \left( \frac{u}{\sqrt{2}2\sigma_0} \right)^{-1/2}du \leq 4 \sigma_0 \sqrt{A_1 R_n}\lesssim \sqrt{\widetilde{\varepsilon}_n}.
\end{multline*}
Since $\widetilde{\varepsilon}_n\rightarrow 0,$ the right-hand side of \eqref{8.28} is of order $\sqrt{\widetilde{\varepsilon}}_n,$ and then $\widetilde{\varepsilon}_n\asymp n^{-1/3}\log n$ is enough to ensure that \eqref{8.28}, or equivalently, formula (8.28) in \cite{geer00}, holds for all $n$ large enough. This completes verification of the conditions in Corollary 8.8 in \cite{geer00}. As a result, cf.\ formula (8.30) in \cite{geer00}, for all $n$ large enough we get the bound
\begin{equation*}
\mathbb{P}_{s_0}\left(  \sup_{g\in\mathcal{G}_{s_0,\widetilde{\varepsilon}_n}}\left| \frac{1}{n} \sum_{i=1}^n \mathcal{W}_i g(z_i) \right| \geq \delta_n \right) \leq C \exp\left( -\frac{n\delta_n^2}{C^2(C_1+1)2R_n^2\sigma_0^2} \right).
\end{equation*}
The right-hand side of this expression converges to zero as $n\rightarrow\infty,$ because $n\widetilde{\varepsilon}_n^2\rightarrow\infty.$ This completes the proof of the lemma.
\end{proof}

\begin{lemma}
\label{lemma1}
Let the conditions of Theorem \ref{mainthm} hold, assume that $n$ is large enough and let $s\in V_{s_0,\widetilde{\varepsilon}_n}.$ Then
\begin{align*}
\frac{1}{2}\frac{1}{n}\sum_{i=1}^n \left\{ \log(1+f_s(z_i)) - f_s(z_i) \right\}
&=-\frac{1}{2}\int_0^1 \frac{( s_0^2(u)-s^2(u) )^2}{s^4(u)}du+O\left( \frac{1}{n} \right)\\
&\geq - \frac{2\mathcal{K}^2}{\kappa^4} \widetilde{\varepsilon}_n^2 +O\left( \frac{1}{n} \right),
\end{align*}
where the remainder term is of order $n^{-1}$ uniformly in $s\in\mathcal{X}.$
\end{lemma}
\begin{proof}
By the elementary inequality $|\log(1+t)-t|\leq t^2$ that is valid for $|t|<1/2,$ we have for all $n$ large enough and uniformly in $s\in V_{s_0,\widetilde{\varepsilon}_n}$ that
\begin{equation*}
\left| \log\left( 1+f_s(z_{i}) \right) - f_s(z_i) \right|\leq f^2_s(z_{i}).
\end{equation*}
Hence
\begin{equation*}
\log\left( 1+f_s(z_{i}) \right) - f_s(z_{i}) \geq - f^2_s(z_{i}),
\end{equation*}
and therefore
\begin{equation*}
\frac{1}{2}\frac{1}{n}\sum_{i=1}^n \left\{ \log(1+f_s(z_i)) - f_s(z_i) \right\} \geq - \frac{1}{2} \frac{1}{n}\sum_{i=1}^{n} f^2_s(z_i).
\end{equation*}
The statement of the lemma now follows by a simple computation employing Assumption \ref{standing} (a) and the Riemann sum approximation of the integral, yielding that for all $n$ large enough,
\begin{align*}
- \frac{1}{2} \frac{1}{n}\sum_{i=1}^{n} f^2_s(z_i)&=-\frac{1}{2}\int_0^1 \frac{( s_0^2(u)-s^2(u) )^2}{s^4(u)}du+O\left( \frac{1}{n} \right)\\
&\geq - \frac{2\mathcal{K}^2}{\kappa^4} \widetilde{\varepsilon}_n^2 +O\left( \frac{1}{n} \right),
\end{align*}
where the remainder term is of order $n^{-1}$ uniformly in $s\in\mathcal{X}.$
\end{proof}

\begin{lemma}
\label{lemma2}
Let the conditions of Theorem \ref{mainthm} hold and let $\varepsilon_n\asymp n^{-1/3}\log n.$ Denote  $\sigma_0^2=2 \ex\left[ \mathcal{W}_i^2 e^{|\mathcal{W}_i|/3} \right].$ There exists a constant $\widetilde{c}_0>0,$ such that
$\widetilde{c}_0 \leq  {\mathcal{K}^4}\sigma_0\left( {\sigma_0} \wedge 4 \right)/{\kappa^4},$
another constant $c_1,$ such that $c_1<\widetilde{c}_0\kappa^2/(2\mathcal{K}^4),$ and a universal constant $C>0,$ for which the inequality 
\begin{multline*}
\pp \left( \sup_{s\in U_{s_0,\varepsilon_n}^c} \prod_{i=1}^n \frac{p_{i,n,s}(Y_{i,n})}{ p_{i,n,s}(Y_{i,n}) } \geq \exp\left( -c_1 n\varepsilon_n^2 \right)\right)\\
\leq C M_{\varepsilon_n} \exp\left( - \frac{ (\widetilde{c}_0\kappa^2/(2\mathcal{K}^4)-c_1)^2  }{8C^2(4\mathcal{K}^2/\kappa^4+1)\sigma_0^2} n\varepsilon_n^2 \right)
\end{multline*}
holds for all $n$ large enough. Here $a \wedge b$ denotes the minimum of two numbers $a$ and $b.$ In particular, as $n\rightarrow\infty,$ the right-hand side of the above display converges to zero.
\end{lemma}

\begin{proof}
As in the proof of Lemma \ref{lemma0}, we will use empirical process theory to establish the result. We use the convention that  the supremum over the empty set is equal to zero. By Assumption \ref{standing} (a), we have $\|s-s_0\|_{2}^2\leq 4 \mathcal{K}^2.$ Hence, using the definition of $M_{\varepsilon_n}$ and $A_{j,\varepsilon_n}$ at the beginning of this appendix, we can write
\begin{multline*}
\pp \left( \sup_{s\in U_{s_0,\varepsilon_n}^c} \prod_{i=1}^n \frac{p_{i,n,s}(Y_{i,n})}{ p_{i,n,s}(Y_{i,n}) } \geq \exp\left( -c_1 n\varepsilon_n^2 \right)\right)\\
=\sum_{j=0}^{M_{\varepsilon_n}} \pp\left( \sup_{s\in A_{j,\varepsilon_n}} \prod_{i=1}^n \frac{p_{i,n,s}(Y_{i,n})}{ p_{i,n,s}(Y_{i,n}) } \geq \exp\left( -c_1 n\varepsilon_n^2 \right) \right).
\end{multline*}
We will individually bound the summands on the right-hand side of the above display, thereby obtaining a bound on its left-hand side, and will show that this upper bound converges to zero as $n\rightarrow\infty.$

Using Lemma \ref{lemma3} ahead (note that the constant $\widetilde{c}_0$ in its statement can be taken arbitrarily small) and recalling the definition of $Z_{i,n,s}(Y_{i,n}),$ $A_{j,\varepsilon_n}$ and $B_{j,\varepsilon_n}$ at the beginning of this appendix, we obtain that for all $n$ large enough
\begin{multline}
\label{pbjn}
\pp\left( \sup_{s\in A_{j,\varepsilon_n}} \prod_{i=1}^n \frac{p_{i,n,s}(Y_{i,n})}{ p_{i,n,0}(Y_{i,n}) } \geq \exp\left( -c_1 n\varepsilon_n^2 \right) \right)\\
\leq \pp\Biggl( \sup_{s\in A_{j,\varepsilon_n}} \exp\left( \sum_{i=1}^n \{ Z_{i,n,s}(Y_{i,n}) - \ex[Z_{i,n,s}(Y_{i,n})]\} \right) \\
 \geq \exp\left( 2^j n\varepsilon_n^2\left( \frac{\widetilde{c}_0\kappa^2}{\mathcal{K}^4} - \frac{\widetilde{C}_0}{2^jn\varepsilon_n^2} - \frac{c_1}{2^{j}} \right)\right) \Biggr)\\
\leq \pp\Biggl( \sup_{s\in B_{j,\varepsilon_n}} \exp\left( \sum_{i=1}^n \{ Z_{i,n,s}(Y_{i,n}) - \ex[Z_{i,n,s}(Y_{i,n})]\} \right) \\
 \geq \exp\left( 2^j n\varepsilon_n^2\left( \frac{\widetilde{c}_0\kappa^2}{\mathcal{K}^4} - \frac{\widetilde{C}_0}{2^j\varepsilon_n^2 n} - \frac{c_1}{2^{j}} \right)\right) \Biggr)\\
\leq\pp\Biggl( \sup_{s\in B_{j,\varepsilon_n}}  \left| \frac{1}{n} \sum_{i=1}^n \mathcal{W}_i f_s(z_i) \right|
 \geq  \delta_n \Biggr),
\end{multline}
where we have set
\begin{equation}
\label{deltan}
\delta_n=\overline{\delta}2^{j+1}\varepsilon_n^2= \left( \frac{\widetilde{c}_0\kappa^2}{\mathcal{K}^4} - \frac{\widetilde{C}_0}{2^j\varepsilon_n^2n} - \frac{c_1}{2^{j}} \right) 2^{j+1}  \varepsilon_n^2.
\end{equation}
Positivity of $\overline{\delta}$ for $n$ large enough is a consequence of the assumptions in the statement of the lemma. We want to apply Corollary 8.8 from \cite{geer00} to the last term in \eqref{pbjn}. In order to do so, we need to verify its conditions, which can be done using arguments similar to those from the proof of Lemma \ref{lemma0} in this Appendix. We first need to find a constant $R_n,$ such that $\sup_{s\in B_{j,\varepsilon_n}} \|f_s\|_{Q_n}\leq R_n.$ We have for all $n$ large enough and all $j=0,1,\ldots,M_{\varepsilon_n},$
\begin{align*}
\frac{1}{n}\sum_{i=1}^n \left\{ \frac{ \int_{z_i}^{z_{i+1}} [s_0^2(u)-s^2(u)]du }{ \int_{z_i}^{z_{i+1}} s^2(u)du } \right\}^2 & = \int_0^1\frac{(s_0^2(u)-s^2(u))^2}{s^4(u)}du\\
&+\Biggl[ \frac{1}{n}\sum_{i=1}^n \left\{ \frac{ \int_{z_i}^{z_{i+1}} [s_0^2(u)-s^2(u)]du }{ \int_{z_i}^{z_{i+1}} s^2(u)du } \right\}^2\\
&- \int_0^1\frac{(s_0^2(u)-s^2(u))^2}{s^4(u)}du\Biggr]\\
&\leq \left(\frac{4\mathcal{K}^2}{\kappa^4}+1\right) 2^{j+1}\varepsilon_n^2,
\end{align*}
where we used Assumption \ref{standing} (a), definition of $B_{j,\varepsilon_n}$ and the assumption that $\varepsilon_n\asymp n^{-1/3}\log n$ to see the last inequality. We can thus take
\begin{equation*}
R_n= \left\{\frac{4\mathcal{K}^2}{\kappa^4}+1\right\}^{1/2} 2^{(j+1)/2}\varepsilon_n.
\end{equation*}
Next, define the constants $K_1,$ $C,$ $C_0$ and $C_1$ as in the proof of Lemma \ref{lemma0}. Since
$\|f_s\|_{\infty}\leq {2\mathcal{K}^2}/{\kappa^2},$
we can take $K_2=2\mathcal{K}^2/\kappa^2.$ We also set $K=4K_1K_2.$ We want that the inequalities $\delta_n\leq C_12R_n^2\sigma_0^2/K,$ $\delta_n\leq 8\sqrt{2}R_n\sigma_0$ and
\begin{equation}
\label{8.28bis}
\sqrt{n}\delta_n \geq C_0 \left( \int_{0}^{\sqrt{2}R_n\sigma_0} H_B^{1/2}\left( \frac{u}{\sqrt{2}\sigma_0},B_{j,\varepsilon_n},Q_n\right) du \vee \sqrt{2} R_n \sigma_0 \right)
\end{equation}
hold. It is not difficult to check by a direct computation that the first two of these inequalities hold with $\delta_n$ as in \eqref{deltan} and $\widetilde{c}_0$ and $c_1$ as in the statement of the lemma. Verification of \eqref{8.28bis}, on the other hand, requires some additional arguments. In order to check \eqref{8.28bis},  we need to show that for all $n$ large enough and all $j=0,1,\ldots,M_{\varepsilon_n},$ the inequalities $n\delta_n^2\geq C_0^2 2 R_n^2\sigma_0^2$ and
\begin{equation}
\label{8.28bisbis}
n\delta_n^2\geq C_0^2 \left( \int_{0}^{\sqrt{2}R_n\sigma_0} H_B^{1/2}\left( \frac{u}{\sqrt{2}\sigma_0},B_{j,\varepsilon_n},Q_n\right) du \right)^2
\end{equation}
hold. It is easy to see that the first of these two inequalities follows from the fact that $n\varepsilon_n^2\rightarrow\infty.$ As far as the second one is concerned, we note that for all $\delta>0$ and for some constant $A>0,$
\begin{equation*}
H_B(\delta,B_{j,\varepsilon_n},Q_n)\leq H_{\infty}\left(\frac{\delta}{2},\mathcal{X}\right)\leq \frac{A}{\delta},
\end{equation*}
where we have used the fact that $B_{j,\varepsilon_n}\subseteq\mathcal{X},$ as well as Lemma 2.1 and Theorem 2.4 from \cite{geer00}. Therefore,
\begin{align*}
\int_{0}^{\sqrt{2}R_n\sigma_0} H_B^{1/2}\left( \frac{u}{\sqrt{2}\sigma_0},B_{j,\varepsilon_n},Q_n\right) du &\leq \sqrt{A}\int_0^{\sqrt{2}R_n\sigma_0}\left( \frac{u}{\sqrt{2}2\sigma_0} \right)^{-1/2}du\\
& = 4{\sqrt{AR_n}}\sigma_0.
\end{align*}
Since
\begin{equation*}
n \overline{\delta}^2 2^{2(j+1)}\varepsilon_n^4 \geq 16 C_0^2 A \sigma_0^2  \left(\frac{4\mathcal{K}^2}{\kappa^4}+1\right)  2^{(j+1)/2}\varepsilon_n
\end{equation*}
for all $n$ large enough and all $j=0,1,\ldots,M_{\varepsilon_n}$ (this follows from the assumption that $\varepsilon_n\asymp n^{-1/3}\log n$), we get that \eqref{8.28bisbis}, and hence \eqref{8.28bis} too, hold. Thus all the assumptions from Corollary 8.8 in \cite{geer00} are satisfied. As a result, the inequality (8.30) from Corollary 8.8 combined with formula \eqref{pbjn} and some further bounding gives that
\begin{multline*}
\pp\left( \sup_{s\in A_{j,\varepsilon_n}} \prod_{i=1}^n \frac{p_{i,n,s}(Y_{i,n})}{ p_{i,n,s}(Y_{i,n}) } \geq \exp\left( -c_1 n\varepsilon_n^2 \right) \right)\\
\leq C \exp\left( - \frac{ (\widetilde{c}_0\kappa^2/(2\mathcal{K}^4) - c_1)^2 }{8 C^2\sigma_0^2(4\mathcal{K}^2/\kappa^4+1)} n\varepsilon_n^2 \right)
\end{multline*}
holds for all $n$ large enough and all $j=0,1,\ldots,M_{\varepsilon_n}.$ The statement of the lemma is an easy consequence of this bound, the fact that $M_{\varepsilon_n}\asymp \log_2 (1/\varepsilon_n)$ for $\varepsilon_n\rightarrow 0$ and the fact that $\varepsilon_n\asymp n^{-1/3}\log n.$
\end{proof}

\begin{lemma}
\label{lemma3} Under the same conditions as in Lemma \ref{lemma2}, there exist two constants $\widetilde{c}_0>0$ and $\widetilde{C}_0>0,$ such that for all $n$ large enough and all $s\in A_{j,\varepsilon_n},$ $j=0,1,\ldots,M_{\varepsilon_n},$ we have
\begin{equation*}
\sum_{i=1}^n \ex[Z_{i,n,s}(Y_{i,n})] \leq - \frac{\widetilde{c}_0\kappa^2}{\mathcal{K}^4}2^j\varepsilon_n^2 n +\widetilde{C}_0.
\end{equation*}
\end{lemma}
\begin{proof}
We have
\begin{align*}
\ex[Z_{i,n,s}(Y_{i,n})] &=\frac{1}{2}\log\left( 1 + \frac{ \int_{z_i}^{z_{i+1}} [s_0^2(u)-s^2(u)]du }{ \int_{z_i}^{z_{i+1}} s^2(u)du } \right)\\
&-\frac{1}{2}\frac{ \int_{z_i}^{z_{i+1}} [s_0^2(u)-s^2(u)]du }{ \int_{z_i}^{z_{i+1}} s^2(u)du } .
\end{align*}
A standard argument shows that for any fixed constant $\overline{C}_0>0,$ there exists another constant $\widetilde{c}_0>0,$ such that for $-1\leq x< \overline{C}_0,$ the inequality $\log(1+x)-x\leq - \widetilde{c}_0 x^2$ holds. Therefore, for all $n$ large enough,
\begin{align*}
\sum_{i=1}^n \ex[Z_{i,n,s}(Y_{i,n})] &\leq - \frac{ \widetilde{c}_0 n }{2}\frac{1}{n} \sum_{i=1}^n \left\{  \frac{ \int_{z_i}^{z_{i+1}} [s_0^2(u)-s^2(u)]du }{ \int_{z_i}^{z_{i+1}} s^2(u)du } \right\}^2\\
&=-\frac{ \widetilde{c}_0 n }{2} \int_0^1 \frac{ (s^2(u)-s^2_0(u))^2 }{s^4(u)}du+O(1)\\
&\leq - \frac{\widetilde{c}_0\kappa^2}{\mathcal{K}^4}2^j\varepsilon_n^2 n+\widetilde{C}_0,
\end{align*}
where we used Assumption \ref{standing} (a) and the definition of $A_{j,\varepsilon_n}.$ Here $\widetilde{C}_0>0$ is some constant independent of a particular $s$ and $n.$ This completes the proof of the lemma.
\end{proof}

\bibliographystyle{plainnat}

\end{document}